\documentclass[12pt,a4paper,english,reqno]{amsart}
\usepackage[a4paper,footskip=1.5em]{geometry}
\usepackage{amsmath,amssymb,amsthm,mathtools}
\usepackage[mathscr]{euscript}
\usepackage[usenames,dvipsnames]{color}
\usepackage{adjustbox,tikz,calc,graphics,babel,standalone}
\usepackage{subcaption}
\usepackage{csquotes,enumerate,verbatim}
\usepackage[final]{microtype}
\usepackage[numbers]{natbib}
\usetikzlibrary{shapes.misc,calc,intersections,patterns,decorations.pathreplacing}
\usepackage{hyperref}
\hypersetup{colorlinks=true,linkcolor=blue,citecolor=blue,pdfpagemode=UseNone,pdfstartview={XYZ null null 1.00}}
\usepackage{cmtiup}

\allowdisplaybreaks

\theoremstyle{plain}
\newtheorem*{theorem*}{Theorem}
\newtheorem{theorem}{Theorem}[section]
\newtheorem{lemma}[theorem]{Lemma}
\newtheorem{claim}[theorem]{Claim}
\newtheorem{proposition}[theorem]{Proposition}
\newtheorem*{claim*}{Claim}

\newtheorem{conjecture}[theorem]{Conjecture}
\newtheorem{problem}[theorem]{Problem}

\theoremstyle{remark}
\newtheorem*{remark}{Remark}

\def\N{\mathbb{N}}
\def\Z{\mathbb{Z}}

\def\P{\mathbb{P}}
\def\E{\mathbb{E}}
\def\C{\mathcal}

\def\DD{D}
\def\JJ{J}
\def\ED{\Lambda}
\def\EDB{{\bar \Lambda}}
\def\CC{10^{40}}

\let\emptyset\varnothing
\let\eps\varepsilon

\let\originalleft\left
\let\originalright\right
\renewcommand{\left}{\mathopen{}\mathclose\bgroup\originalleft}
\renewcommand{\right}{\aftergroup\egroup\originalright}

\makeatletter
\def\imod#1{\allowbreak\mkern10mu({\operator@font mod}\,\,#1)}
\makeatother

\begin{document}

\title{Reconstructing random jigsaws}

\author{Paul Balister}
\address{Department of Mathematical Sciences, University of Memphis, Memphis TN 38152, USA}
\email{pbalistr@memphis.edu}

\author{B\'{e}la Bollob\'{a}s}
\address{Department of Pure Mathematics and Mathematical Statistics, University of Cambridge, Wilberforce Road, Cambridge CB3\thinspace0WB, UK, {\em and\/}
Department of Mathematical Sciences, University of Memphis, Memphis TN 38152, USA, {\em and\/} London Institute for Mathematical Sciences, 35a South St., Mayfair, London W1K\thinspace2XF, UK}
\email{b.bollobas@dpmms.cam.ac.uk}

\author{Bhargav Narayanan}
\address{Department of Pure Mathematics and Mathematical Statistics, University of Cambridge, Wilberforce Road, Cambridge CB3\thinspace0WB, UK}
\email{b.p.narayanan@dpmms.cam.ac.uk}

\date{30 May 2017}

\begin{abstract}
A colouring of the edges of an $n \times n$ grid is said to be \emph{reconstructible} if the colouring is uniquely determined by the multiset of its $n^2$ \emph{tiles}, where the tile corresponding to a vertex of the grid specifies the colours of the edges incident to that vertex in some fixed order. In 2015, Mossel and Ross asked the following question: if the edges of an $n \times n$ grid are coloured independently and uniformly at random using $q=q(n)$ different colours, then is the resulting colouring reconstructible with high probability? From below, Mossel and Ross showed that such a colouring is not reconstructible when $q = o(n^{2/3})$ and from above, Bordenave, Feige and Mossel and Nenadov, Pfister and Steger independently showed, for any fixed $\eps > 0$, that such a colouring is reconstructible when $q \ge n^{1+\eps}$. Here, we improve on these results and prove the following: there exist absolute constants $C, c > 0$ such that, as $n \to \infty$, the probability that a random colouring as above is reconstructible tends to $1$ if $q \ge Cn$ and to $0$ if $q \le cn$.
\end{abstract}

\maketitle

\section{Introduction}
The reconstruction problem for a family of discrete structures asks the following: is it possible to uniquely reconstruct a structure in this family from the `deck' of all its substructures of some fixed size? Combinatorial reconstruction problems have a very rich history. The oldest such problem is perhaps the graph reconstruction conjecture of Kelly and Ulam~\citep{kelly,ulam,harary}, and analogous questions for various other families of discrete structures have since been studied; see, for instance, the results of Alon, Caro, Krasikov and Roddity~\citep{sets} on reconstructing finite sets satisfying symmetry conditions, Pebody's~\citep{groups, necklaces} results on reconstructing finite abelian groups, and the results of Pebody, Radcliffe and Scott~\citep{plane} on reconstructing finite subsets of the plane.

Another natural line of enquiry, and the one we pursue here, is to ask how the answer to the reconstruction problem changes when we are required to reconstruct a \emph{typical} (as opposed to an arbitrary) structure in a family of discrete structures. These probabilistic questions typically have substantially different answers as compared to their extremal counterparts, as evidenced by the results of Bollob\'as~\citep{Bela} and Radcliffe and Scott~\citep{Alex}, for example.

\begin{figure}
\begin{center}
\begin{tikzpicture}[scale = 1.3]

\foreach \x in {0,1,2,3,4}
\foreach \y in {0,1,2,3,4}
	\node [scale=0.3] (l\x\y) at (\x, \y) {};	

\foreach \x in {0,1,2,3,4}
\foreach \y in {0,1,2,3,4}
	\node [scale=0.3] (r\x\y) at (\x+5, \y) {};	

\draw [very thick, blue](l11) -- (l12);
\draw [very thick, red](l11) -- (l21);
\draw [very thick, green](l11) -- (l01);
\draw [very thick, yellow](l11) -- (l10);

\draw [very thick, brown](l21) -- (l22);
\draw [very thick, blue](l21) -- (l31);
\draw [very thick, cyan](l21) -- (l20);

\draw [very thick, black](l31) -- (l32);
\draw [very thick, magenta](l31) -- (l41);
\draw [very thick, green](l31) -- (l30);

\draw [very thick, cyan](l12) -- (l13);
\draw [very thick, black](l12) -- (l22);
\draw [very thick, brown](l12) -- (l02);

\draw [very thick, red](l22) -- (l23);
\draw [very thick, yellow](l22) -- (l32);

\draw [very thick, cyan](l32) -- (l33);
\draw [very thick, green](l32) -- (l42);

\draw [very thick, green](l13) -- (l03);
\draw [very thick, red](l13) -- (l14);
\draw [very thick, magenta](l13) -- (l23);

\draw [very thick, black](l23) -- (l24);
\draw [very thick, green](l23) -- (l33);

\draw [very thick, brown](l33) -- (l34);
\draw [very thick, yellow](l33) -- (l43);

\draw [very thick, blue](r11) -- (r12);
\draw [very thick, red](r11) -- (r21);
\draw [very thick, green](r11) -- (r01);
\draw [very thick, yellow](r11) -- (r10);

\draw [very thick, brown](r21) -- (r22);
\draw [very thick, blue](r21) -- (r31);
\draw [very thick, cyan](r21) -- (r20);

\draw [very thick, black](r31) -- (r32);
\draw [very thick, magenta](r31) -- (r41);
\draw [very thick, green](r31) -- (r30);

\draw [very thick, cyan](r12) -- (r13);
\draw [very thick, black](r12) -- (r22);
\draw [very thick, brown](r12) -- (r02);

\draw [very thick, red](r22) -- (r23);
\draw [very thick, yellow](r22) -- (r32);

\draw [very thick, cyan](r32) -- (r33);
\draw [very thick, green](r32) -- (r42);

\draw [very thick, green](r13) -- (r03);
\draw [very thick, red](r13) -- (r14);
\draw [very thick, magenta](r13) -- (r23);

\draw [very thick, black](r23) -- (r24);
\draw [very thick, green](r23) -- (r33);

\draw [very thick, brown](r33) -- (r34);
\draw [very thick, yellow](r33) -- (r43);

\foreach \x in {5,6,7,8,9}
\foreach \y in {0,1,2,3,4}
	\fill [white] (\x-0.1,\y+0.4) rectangle (\x+0.1,\y+0.6);

\foreach \x in {5,6,7,8,9}
\foreach \y in {0,1,2,3,4}
	\fill [white] (\x+0.4,\y-0.1) rectangle (\x+0.6,\y+0.1);

\fill [white] (0,-0.1) rectangle (5,0.4);
\fill [white] (0,-0.1) rectangle (0.4,4);
\fill [white] (0,3.6) rectangle (5,4);
\fill [white] (3.6,-0.1) rectangle (4,4);

\fill [white] (5,-0.1) rectangle (9,0.5);
\fill [white] (5,-0.1) rectangle (5.5,4);
\fill [white] (5,3.5) rectangle (9,4);
\fill [white] (8.5,-0.1) rectangle (9,4);

\foreach \x in {1,2,3}
\foreach \y in {1,2,3}
	\node at (\x, \y) [inner sep=0.7mm, thick, circle, draw=black!100, fill=black!0] {};	

\foreach \x in {1,2,3}
\foreach \y in {1,2,3}
	\node at (\x+5, \y) [inner sep=0.7mm, thick, circle, draw=black!100, fill=black!0] {};
	
\node at (2,-0.15) {$J$};	
\node at (7,-0.15) {$\DD(J)$};
\end{tikzpicture}
\end{center}
\caption{A $(3,8)$-jigsaw and its deck.}
\label{edges}
\end{figure}
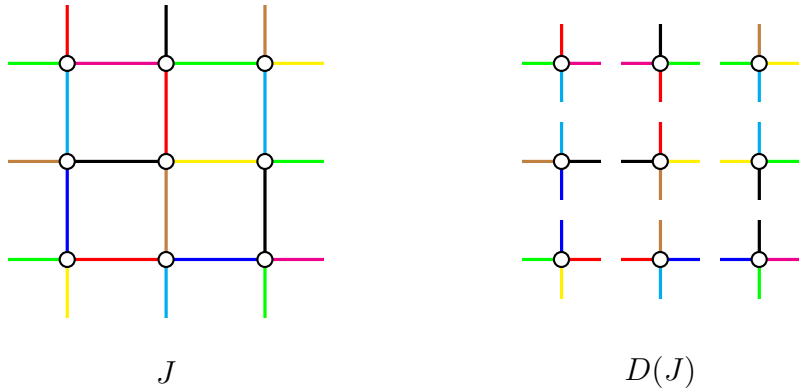

Here, we shall study a reconstruction problem proposed by Mossel and Ross in connection with the problem of shotgun sequencing DNA sequences. To state this problem, we need a few definitions. 

For $n \in \N$, we write $[n]$ for the set $\{1, 2, \dots, n\}$, and by the \emph{extended} $n\times n$ grid, we mean the grid $[n]^2\subset \Z^2$ together with the edges of ${\mathbb Z}^2$ incident to the boundary vertices. For $n,q\in \N$, an \emph{$(n,q)$-jigsaw} is a $q$-coloured extended $n \times n$ grid, i.e., an extended $n \times n$ grid whose $2n(n+1)$ edges are coloured using a set of $q$ different colours which we take to be $[q]$ for concreteness. The \emph{tile} of an $(n,q)$-jigsaw corresponding to a vertex $v\in [n]^2$ is given by the colouring of the four edges incident to $v$; more precisely, writing $e_1 = (0,1)$, $e_2 = (1,0)$, $e_3=-e_1$ and $e_4=-e_2$, if the edge between $v$ and $v+e_i$ gets colour $c_i\in [q]$ for $1\le i \le 4$, then the tile corresponding to $v$ is the tuple $(c_i)_{i=1}^4\in [q]^4$. Finally, the \emph{deck} of an $(n,q)$-jigsaw is the multiset of the tiles of the jigsaw, one for each vertex of  $[n]^2$.

We now define what it means for a jigsaw to be reconstructible from its deck.  Writing ${\mathcal J}(n,q)$ for the set of all $(n,q)$-jigsaws and ${\mathcal D}(n,q)$ for the family of all multisets of size $n^2$ whose elements are chosen from $[q]^4$, let $D\colon{\mathcal J}(n,q) \to {\mathcal D}(n,q)$ be the map sending a jigsaw $J$ to its deck $D(J)$. We say that a jigsaw $J \in {\mathcal J}(n,q)$ is \emph{reconstructible} if $D^{-1} (D (J))=\{J\}$; equivalently, a jigsaw $J$ is reconstructible if $D(J)=D(J')$ implies $J=J'$.

We view ${\mathcal J}(n,q)$ as a probability space by endowing it with the uniform distribution, and write $J({n,q})$ for a random $(n,q)$-jigsaw drawn from this distribution; equivalently, $J(n,q)$ is a random $(n,q)$-jigsaw generated by independently colouring each edge of the extended $n \times n$ grid with a randomly chosen element of $[q]$. Our primary concern is the following problem about the reconstructibility of a random $(n,q)$-jigsaw raised by Mossel and Ross~\citep{mossel}; of course, there exists only one $(n,1)$-jigsaw for each $n \in \N$ (and this jigsaw is trivially reconstructible), so in what follows, we assume implicitly that $q \ge 2$.

\begin{problem}
For what $q = q(n)$ is $\JJ(n,q)$ reconstructible with high probability?
\end{problem}
From below, Mossel and Ross~\citep{mossel} showed that $\P(\JJ(n,q) \text{ is reconstructible}) \to 0$ when $q = o(n^{2/3})$ due to the presence of local obstacles to reconstruction: in this regime, a random $(n,q)$-jigsaw contains, with high probability, two configurations each consisting of two neighbouring vertices which may be `exchanged' in the jigsaw, and this is easily seen to obstruct unique reconstruction; however, this argument does not extend to configurations involving single exchangeable vertices (and to a corresponding bound when $q = o(n)$) since the presence of two identical tiles in the deck does not necessarily prevent unique reconstruction. From above, Bordenave, Feige and Mossel~\citep{feige} and Nenadov, Pfister and Steger~\citep{steger} independently showed, for any fixed $\eps > 0$, that $\P(\JJ(n,q) \text{ is reconstructible}) \to 1$ when $q \ge n^{1+\eps}$. Here, we improve on both of these bounds and prove the following nearly optimal result.

\begin{theorem}\label{reconthm}
There exist absolute constants $C,c > 0$ such that, as $n \to \infty$, we have
\[
\P \left( \JJ(n,q) \text{ is reconstructible} \right) \to
\begin{cases}
1 &\mbox{if } q \ge Cn, \text{ and}\\
0 &\mbox{if } 2\le q \le cn.\\
\end{cases}
\]
\end{theorem}

The two results contained in the statement of Theorem~\ref{reconthm} are proved by very different methods: the `$0$-statement' follows from a double counting argument, while the proof of the `$1$-statement' is based on an isoperimetric argument which draws from (but is somewhat more involved than) the strategy used by Bordenave, Feige and Mossel~\citep{feige} where one attempts to reconstruct a suitably large neighbourhood of a tile in order to identify its neighbours in the jigsaw.

We shall prove Theorem~\ref{reconthm} with $C = \CC$ and $c = 1/\sqrt{e}$. With some more effort, it should be possible to refine our proof of Theorem~\ref{reconthm} to show that the result holds for any $C>1$ (at which point our argument breaks down); however, we choose not to present the details of this stronger claim because we believe the critical number of colours for an $n \times n$ grid to be $n / \sqrt{e}$, and conjecture that the $0$-statement in Theorem~\ref{reconthm} is sharp.

\begin{conjecture}\label{mainconj}
For any $\eps > 0$, as $n \to \infty$, we have
\[
\P \left( \JJ(n,q) \text{ is reconstructible} \right) \to 1
\]
for all $q \ge (1/\sqrt{e} +\eps)n$.
\end{conjecture}

This paper is organised as follows. We begin with some notation and preliminary discussion in Section~\ref{s:prelim}. We give the short proof of the $0$-statement in Theorem~\ref{reconthm} in Section~\ref{s:lower}. We prove the key estimate required for the proof of the $1$-statement in Theorem~\ref{reconthm} in Section~\ref{s:nhd}, and complete the proof of our main result in Section~\ref{s:upper}. We conclude with some discussion in Section~\ref{s:conc}.

\begin{remark}After the results in this paper were proved (in November 2016), but before this paper was completed, Martinsson~\citep{mart}, working independently, also announced (in January 2017) a proof of a result analogous to Theorem~\ref{reconthm} in a very closely related model (and with a more reasonable constant in the $1$-statement). We briefly point out that while the respective $0$-statements are established in essentially the same fashion both here and in~\citep{mart}, the estimates needed to prove the respective $1$-statements are established by quite different approaches.
\end{remark}
\section{Preliminaries}\label{s:prelim}
For a pair of integers $a \le b$, we write $[a,b]$ for the set $\{a, a+1, \dots, b\}$, and for a natural number $n\in \N$, we abbreviate the set $[1,n]$ by $[n]$.

We define the vectors $e_1 = (0,1)$, $e_2 = (1,0)$, $e_3 = -e_1$ and $e_4 = -e_2$, and we endow the square lattice $\Z^2$ with the graph structure of the infinite grid where two vertices $u, v \in \Z^2$ are adjacent if $u - v = e_i$ for some $1 \le i \le 4$; also, we write $\ED$ for the set of edges of the infinite grid on $\Z^2$.

Let $X \subset \Z^2$ be a finite subset of the square lattice. We write $\ED(X) \subset \ED$ for the set of edges of the grid induced by $X$ and $\partial X \subset \ED$ for the \emph{boundary} of $X$, i.e., the set of edges between between $X$ and $\Z^2 \setminus X$; also, we write $\EDB(X) = \ED(X) \cup \partial X$ for the set of edges of the grid with at least one endpoint in $X$. Since $X$ is finite, note that $\Z^2 \setminus X$ contains a unique infinite connected component; the \emph{external boundary} of $X$, written $\partial_e X$, is the set of edges between $X$ and this infinite component, and the \emph{internal boundary} of $X$, written $\partial_i X$, is defined to be $\partial X \setminus \partial_e X$. Finally, the \emph{vertex boundary} of $X$ is defined to be the set of vertices of $X$ incident to some edge of $\partial_e X$.

Observe that if the points of a finite set $X \subset \Z^2$ have $a$ different $x$-coordinates and $b$ different $y$-coordinates in total, then $|X|\le ab$ and the external boundary of $X$ has size at least $2a+2b$; this observation implies the following well-known isoperimetric statement.

\begin{proposition}\label{p:iso}
For any finite set $X \subset \Z^2$, we have $|\partial X | \ge |\partial_e X| \ge 4 |X|^{1/2}$. \qed
\end{proposition}

We say that a finite set $X \subset \Z^2$ is \emph{connected} if it is connected when viewed as a subset of the vertex set of the infinite grid, and in what follows, the distance between two points $u, v \in \Z^2$ will always mean the graph-distance between $u$ and $v$ in the infinite grid. Also, we say that a finite set of edges $A \subset \ED$ is \emph{dual-connected} if the corresponding set of edges in the planar dual of the infinite grid is connected. Finally, for $X \subset \Z^2$ and $A \subset \ED$, we write $D(X,A)$ for the graph on $X$ whose edge set is $\ED(X) \setminus A$; in other words, $D(X,A)$ is the graph induced by $X$ in the grid after we delete the edges in $A$.

It will be convenient to have some notation to deal with maps from $\Z^2$ to $\Z^2$. Let $f$ be an injective map from a finite set $X \subset \Z^2$ to $\Z^2$. We say that a set $Y \subset X$ is \emph{$f$-rigid} if $f(x) - f(y) = x - y$ for all $x,y \in Y$. A \emph{block} of $f$ is a maximal connected rigid subset of $X$; it is easy to check that each vertex of $X$ belongs to a unique block, so the blocks of $f$ partition $X$. An edge $z\in \ED(X)$ is said to be \emph{$f$-split} if the endpoints of $z$ belong to different blocks of $f$. We write $\ED_f \subset \ED$ for the union of $\partial X$ and the set of $f$-split edges; loosely speaking, $\ED_f$ is the set of those edges across which we cannot `control' $f$. Note that $\ED_f$ may be decomposed into dual-connected components; the following geometric fact about such components will prove useful.

\begin{proposition}\label{block}
Let $f$ be an injective map from a finite set $X \subset \Z^2$ to $\Z^2$ and let $A$ be a dual-connected component of $\ED_f$. If $Y \subset X$ is a connected component of $D(X,A)$, then the vertices of $Y$ incident to some edge of $A$ are all contained in a single block of $f$. \qed
\end{proposition}

We will also need the following property of finite grids.
\begin{proposition}\label{p:3sep}
	For $n\in \N$, if $X \subset [n]^2$ is such that the distance between any pair of distinct vertices in $X$ is at least three, then $[n]^2 \setminus X$ is connected. \qed	
\end{proposition}

Finally, let us quickly restate the problem at hand formally. Note that the edges of the extended $n \times n$ grid are precisely the elements of the set $\EDB([n]^2)$, so an $(n,q)$-jigsaw $J$ is a map $J\colon \EDB([n]^2) \to [q]$. Given an $(n,q)$-jigsaw $J$, the tile $J_v$ corresponding to a vertex $v \in [n]^2$ is the sequence $(J(v,v+e_i))_{i=1}^4 \in [q]^4$, and the deck $\DD(J)$ of $J$ is the multiset $\{ J_v: v \in [n]^2 \}$. As defined previously, a jigsaw $J$ is reconstructible from its deck if $\DD(J') = \DD(J)$ implies that $J' = J$. We write $\JJ(n,q)$ to denote a random $(n,q)$-jigsaw generated by independently colouring each edge of $\EDB([n]^2)$ with a randomly chosen element of $[q]$. In this language, our primary concern is the following question: for what $q = q(n)$ is $\JJ(n,q)$ reconstructible with high probability?

We shall make use of standard asymptotic notation; in what follows, the variable tending to infinity will always be $n$ unless we explicitly specify otherwise. We use the term \emph{with high probability} to mean with probability tending to $1$ as $n \to \infty$. For the sake of clarity of presentation, we systematically omit floor and ceiling signs whenever they are not crucial.

\section{Proof of the $0$-statement}\label{s:lower}

In this short section, we prove the $0$-statement in Theorem~\ref{reconthm} by an elementary counting argument.
\begin{proof}[Proof of the $0$-statement in Theorem~\ref{reconthm}]
Recall that ${\mathcal J}(n,q)$ is the set of all $(n,q)$-jigsaws, ${\mathcal D}(n,q)$ is the family of all multisets of size $n^2$ whose elements are chosen from $[q]^4$, and $D\colon{\mathcal J}(n,q) \to {\mathcal D}(n,q)$ is the map sending a jigsaw $J$ to its deck $D(J)$. 

Let ${\mathcal J}_R(n,q) \subset {\mathcal J}(n,q)$ denote the set of all reconstructible jigsaws, i.e., jigsaws $J$ such that $D^{-1} (D (J))=\{J\}$. Since $D\colon{\mathcal J}_R(n,q) \to {\mathcal D}(n,q)$ is an injection, $|{\mathcal J}_R(n,q)| \le |{\mathcal D}(n,q)|$. Consequently, we have
\[ \P(\JJ(n,q) \text{ is reconstructible}) = |{\mathcal J}_R(n,q)| / |{\mathcal J}(n,q)| \le |{\mathcal D}(n,q)| /  |{\mathcal J}(n,q)|.\]
Now, it is easy to see that
\[
|{\mathcal D}(n,q)|= \binom{n^2+q^4-1}{n^2} \hspace{10pt} \text{and} \hspace{10pt}  |{\mathcal J}(n,q)|=q^{2n(n+1)},
\]
so it follows that
\[ \P(\JJ(n,q) \text{ is reconstructible}) \le \binom{n^2 + q^4 - 1}{n^2 - 1} q^{-2n^2 - 2n} \le \binom{n^2 + q^4}{n^2} q^{-2n^2 - 2n} .\]

If $2 \le q \le \sqrt{n}$, then we have
\[
\binom{n^2 + q^4}{n^2} q^{-2n^2 - 2n}\le \binom{2n^2}{n^2} 2^{-2n^2 - 2n} \le 2^{-2n}.
\]
If $\sqrt{n} < q \le n/\sqrt{e}$ on the other hand, then we deduce using Stirling's approximation that
\begin{align*}
\binom{n^2 + q^4}{n^2}   q^{-2n^2 - 2n}
&= \frac{q^{2n^2 - 2n}}{(n^2)!} \prod_{i=1}^{n^2} \left(1+\frac{i}{q^4}\right) \le\frac{q^{2n^2-2n}}{(n^2)!}\left(1+\frac{n^2}{q^4}\right)^{n^2}\\
&=O\left(\frac{q^{-2n}}{n}\exp\left( n^2\log\left(\frac{q^2}{n^2}\right) + \frac{n^4}{q^4} + n^2\right)\right)=O\left(q^{-2n}\right).
\end{align*}

We conclude from the above estimates that
\[\P(\JJ(n,q) \text{ is reconstructible}) = o(1)\]
for all $2 \le q \le n/\sqrt{e}$.
\end{proof}

\section{Reconstructing large neighbourhoods}\label{s:nhd}
The starting point of our approach to proving the $1$-statement in Theorem~\ref{reconthm} is the strategy adopted by Bordenave, Feige and Mossel~\citep{feige} to show that $\JJ = \JJ(n,q)$ is reconstructible with high probability when $q \ge n^{1+\eps}$ for some fixed $\eps > 0$. Given the deck $\DD(\JJ)$ of $\JJ$, Bordenave, Feige and Mossel use the following procedure to identify the neighbours of a given tile $\JJ_v$ with $v \in [n]^2$. For some large integer $k \approx 1/\eps$, they consider all subsets of $\DD(\JJ)$ of size $(2k+1)^2$ that include the tile $\JJ_v$ and for each such set, they check if the tiles in that set can be `legally assembled' on a $(2k+1) \times (2k+1)$ grid with $\JJ_v$ at the centre of this grid. While there might exist many such legal assemblies with $\JJ_v$ at the centre, they show that with high probability, the four neighbours of $\JJ_v$ in any such legal assembly are identical to the four tiles neighbouring $\JJ_v$ in the original jigsaw. This allows them to identify the neighbours of all tiles corresponding to vertices at distance at least $k$ from the boundary of the grid; once this has been accomplished, it is reasonably straightforward to reconstruct $\JJ$.

We adopt a similar strategy to the one described above, although in order to show that $\JJ(n,q)$ is reconstructible when $q \approx n$ (as opposed to when $q \ge n^{1+\eps}$), we require more delicate arguments; for example, we need to take $k \approx \log n$ (as opposed to $k \approx 1/\eps$) and this in turn necessitates more careful estimates.

We now fix positive integers $n,q \in \N$ and set $k = k(n) = \lceil \log n \rceil$; all inequalities in the sequel will hold provided $n$ and $k$ are sufficiently large.

\subsection{Constraint graphs}
Let $J\colon \EDB([n]^2) \to [q]$ be an $(n,q)$-jigsaw, and  let $f$ be an injection from a finite set $X \subset \Z^2$ to $[n]^2$. We say that \emph{$f$ is feasible for $J$} if for any pair of adjacent vertices $x,y \in X$, we have $J(x',x'+ y -x) = J(y', y'+x - y)$, where $x' = f(x)$ and $y' = f(y)$. Clearly, any injective function $f$ as above describes an arrangement of a subset of the tiles of $J$ on the grid at the vertices of $X$ (where the tile placed at a position $x \in X$ is precisely $J_{f(x)}$); our definition of feasibility makes precise the notion of when $f$ describes a legal arrangement of tiles. Constraint graphs provide us with an alternate description of legal arrangements and we define these objects below.

The \emph{constraint graph} of an injective map $f$ from a finite set $X \subset \Z^2$ to $\Z^2$, denoted by $\C{G}_f$, is a graph whose vertex set is a subset of $\ED$ and whose edge set contains one edge, called a \emph{constraint}, for each $f$-split edge, where if $\{x,y\}$ is an $f$-split edge with $y = x + e_i$ for some $1 \le i \le 4$, then the constraint corresponding to this edge is an edge joining $\{f(x), f(x) + e_i\}$ and $\{f(y), f(y) - e_i\}$ in the constraint graph; the vertex set of $\C{G}_f$ is the subset of $\ED$ spanned by the edges of $\C{G}_f$. In the language of constraint graphs, it is clear that if $J\colon \EDB([n]^2) \to [q]$ is an $(n,q)$-jigsaw and $f$ is an injection from a finite subset of $\Z^2$ to $[n]^2$, then $f$ is feasible for $J$ if and only if $J$ is constant on each connected component of $\C{G}_f$. We define $\gamma(f)$ to be the difference between the size of the vertex set of $\C{G}_f$ and the number of connected components of $\C{G}_f$. We require the following observation due to Bordenave, Feige and Mossel~\citep{feige}; we include the short proof for completeness.

\begin{proposition}\label{f-cost}
For any injective map $f$ from a finite subset of $\Z^2$ to $[n]^2$, we have
\[ \P(f \text{ is feasible for } \JJ(n,q)) = q^{-\gamma(f)}.\]
\end{proposition}
\begin{proof}
First, choose a representative from each connected component of $\C{G}_f$. It is clear that $f$ is feasible for $\JJ = \JJ(n,q)$ if and only if the following holds: for each vertex of $\C{G}_f$, the colour assigned by $\JJ$ to this vertex is equal to the colour assigned by $\JJ$ to the representative vertex from the corresponding connected component of $\C{G}_f$. Thus, the event that $f$ is feasible for $\JJ$ is an intersection of $\gamma(f)$ independent events, and each of these events has probability $1/q$; the claim follows.
\end{proof}

It is easy to see that the maximum degree of a constraint graph is at most two, so every constraint graph is a union of paths and cycles; this observation implies the following.

\begin{proposition}\label{gamma-bound}
If $f$ is an injection from a finite subset of $\Z^2$ to $\Z^2$, then $\gamma(f) \ge |V(\C{G}_f)|/2 \ge |E(\C{G}_f)|/2$. \qed
\end{proposition}

\subsection{Windows}
To make precise the idea of recovering the four tiles neighbouring a given tile by attempting to reconstruct a large neighbourhood the tile in question, we need the notion of a `window'.

For $v \in [n]^2$ and an $(n,q)$-jigsaw $J$, a \emph{$v$-window} with respect to $J$ is an injective map $f\colon [-k,k]^2 \to [n]^2$ such that $f(0,0) = v$ and $f$ is feasible for $J$; we remind the reader that $k = \lceil \log n \rceil$ here, and in what follows.

If $v \in [n]^2$ is at distance at least $k$ from the vertex boundary of the $n \times n$ grid, then the map defined by $f(x) = v + x$ for all $x \in [-k,k]^2$ is a $v$-window; more generally, if there exists some $v' \in [n]^2$ at distance at least $k$ from the vertex boundary of the $n \times n$ grid such that $J_{v'} = J_{v}$, then the map defined by $f(0,0) = v$ and $f(x) = v' + x$ for all $x \in [-k,k]^2 \setminus \{(0,0)\}$ is a $v$-window. A $v$-window $f$ is said to be \emph{trivial} if $(J_{f(e_i)})_{i = 1}^4 = (J_{v' + e_i})_{i = 1}^4$ for some $v' \in [n]^2$ such that $J_{v'} = J_v$; in other words, a $v$-window is trivial if the four tiles neighbouring $J_v$ in the $v$-window are identical to the four tiles neighbouring some tile $J_{v'}$ in the jigsaw, with $J_{v'}$ itself identical to $J_v$. This definition of triviality is motivated by the fact that when $q \approx n$, the deck of $\JJ(n,q)$ may contain some tiles of multiplicity greater than one (though, as we shall see, this will not present an obstacle to reconstruction). We shall show, provided $q$ is suitably large, that all windows with respect to $\JJ(n,q)$ are trivial with high probability; the aim of this section is to establish the following lemma.

\begin{lemma}\label{lnhd}
If $q \ge \CC n$, then $\JJ(n,q)$ has the following property with high probability: for each $v \in [n]^2$, every $v$-window with respect to $\JJ(n,q)$ is trivial.
\end{lemma}

\subsection{Templates}
To prove Lemma~\ref{lnhd}, it is natural to first attempt to use a union bound over all candidate injective maps from $[-k,k]^2$ to $[n]^2$; however, this turns out to be too crude for our purposes. The reason for this is roughly as follows: the number of candidate windows is artificially inflated by maps $f\colon [-k,k]^2 \to [n]^2$ with a large number of `holes'; more precisely, there exist too many candidate windows $f\colon [-k,k]^2 \to [n]^2$ with the property that one of the blocks of $f$ is contained entirely in the interior of another block of $f$. One could hope to address this issue by locally modifying a candidate window so as to remove such pairs of `nested blocks', but attempting to do so results in a situation where some tiles of the jigsaw end up getting used multiple times.

To circumvent the difficulties outlined above, we introduce the notion of a `template'. To introduce this notion, it will be helpful to first have some notation.

Let $A \subset \EDB([-k,k]^2)$ be a set of edges of the grid. Recall that $D([-k,k]^2, A)$ is the graph on $[-k,k]^2$ whose edge set is $\ED([-k,k]^2) \setminus A$. For any connected component $X \subset [-k,k]^2$ of $D([-k,k]^2, A)$, we define the \emph{quasiblock $\hat X$} associated with $X$ to be the set of vertices of $X$ incident to some edge in $A$; in the sequel, when we refer to a quasiblock $\hat X$ of $A$, we implicitly assume that the corresponding connected component of $D([-k,k]^2, A)$ is denoted by $X$. Finally, we write $A^*$ for the set $A \cap \ED([-k,k]^2)$.

For $v \in [n]^2$, a \emph{$v$-template} is a pair $(A,h)$, where $A \subset \EDB([-k,k]^2)$ and $h$ is an injective map from the union of the quasiblocks of $A$ to $[n]^2$, such that

\begin{enumerate}
	\item $A$ contains at least one edge incident to $(0,0)$,
	\item $A$ does not consist of precisely the four edges incident either to $(0,0)$ or one of its four neighbours,
	\item $A$ is dual-connected,
	\item $h(0,0) = v$,
	\item either $\partial [-k,k]^2 \subset A$ or $\partial [-k,k]^2 \cap A = \emptyset$,
	\item each quasiblock of $A$ is $h$-rigid, and
	\item each edge of $A^*$ is $h$-split.
\end{enumerate}

Given an $(n,q)$-jigsaw $J$, we abuse notation slightly and say that a $v$-template $(A,h)$ is feasible for $J$ if $h$ is feasible for $J$. The definition of a template is motivated by the following fact.

\begin{proposition}\label{template}
Let $J$ be an $(n,q)$-jigsaw and let $v \in [n]^2$. If there exists a nontrivial $v$-window $f$ with respect to $J$, then there exists a $v$-template $(A,h)$ that is feasible for $J$.
\end{proposition}
\begin{proof}
Since any tile is uniquely determined by its four neighbours in any valid arrangement of tiles, it is easy to check using the fact that $f$ is a nontrivial $v$-window that there exists an $f$-split edge $z$ incident to $(0,0)$ with the property that the dual-connected component of $z$ in $\ED_f$ does not consist of precisely the four edges incident either to $(0,0)$ or one of its four neighbours. We now take $A$ to be the dual-connected component of $z$ in $\ED_f$ and $h$ to be the restriction of $f$ to the endpoints of $A$ in $[-k,k]^2$.

Clearly, $A$ contains at least one edge incident to $(0,0)$, does not consist of precisely the four edges incident either to $(0,0)$ or one of its four neighbours, and is dual-connected. As $f$ is a $v$-window that extends $h$, we have $h(0,0) = v$. Next, since $A$ is a dual-connected component of $\ED_f$ and $\partial [-k,k]^2$ is a dual-connected subset of $\ED_f$, either $\partial [-k,k]^2 \subset A$ or $\partial [-k,k]^2 \cap A = \emptyset$. Furthermore, it follows from Proposition~\ref{block} that every quasiblock of $A$ is a subset of a single block of $f$; since $f$ extends $h$, it follows that every quasiblock of $A$ is $h$-rigid. Finally, since each edge of $A^*$ is $f$-split, each edge of $A^*$ must also be $h$-split.
\end{proof}

We shall prove Lemma~\ref{lnhd} using a union bound over templates as opposed to windows; in particular, we shall show, provided $q$ is suitably large, that with sufficiently high probability, no $v$-template $(A,h)$ is feasible for $\JJ(n,q)$.

We say that a template $(A,h)$ is \emph{large} if $\partial [-k,k]^2 \subset A$, and \emph{small} if $A \cap \partial [-k,k]^2 = \emptyset$. Of course, every template is either large or small. We shall require slightly different arguments to deal with large and small templates. The following fact will prove useful when estimating the number of templates of both types; see Problem~45 in~\citep{coffee}, for instance.

\begin{proposition}\label{component}
	In a graph of maximal degree $\Delta$, the number of connected induced subgraphs with $l + 1$ vertices, one of which is a given vertex, is at most $ (e(\Delta - 1))^l$. \qed
\end{proposition}

\subsection{Large templates}
We will need an estimate for the number of large templates, as well as an estimate for the probability that such a template is feasible for $\JJ(n,q)$.

In order to simplify our bookkeeping, it will be helpful to introduce the notion of a `cluster'. Let $(A,h)$ be a large $v$-template. For a quasiblock $\hat X \subset  [-k,k]^2$ of $A$, let $h(\hat X) \subset [n]^2$ denote the (rigid) image of $\hat X$ under $h$. Let us define the \emph{cluster graph} of $(A,h)$ to be the graph on the quasiblocks of $A$ where two quasiblocks $\hat X$ and $\hat Y$ are adjacent if there exists an edge of the lattice between $h(x)$ and $h(y)$ for some $x \in \hat X$ and $y \in \hat Y$ and furthermore, this edge belongs to the external boundary of both $h(\hat X)$ and $h(\hat Y)$. A \emph{cluster} of $(A,h)$ is then a subset of $[n]^2$ consisting of the images of all the quasiblocks in a connected component of the cluster graph.

For non-negative integers $\delta$, $r_1$ and $r_2$, we say that a large $v$-template $(A,h)$ is of type $(\delta, r_1, r_2)$ if $|A| = \delta$, the number of quasiblocks of $A$ is $r_1 + r_2$, and the number of clusters of $(A,h)$ is $r_1$. Writing $N_l(\delta, r_1, r_2)$ for the number of large  $v$-templates of type $(\delta, r_1, r_2)$, we have the following estimate.

\begin{proposition}\label{ltemp-count}
For non-negative integers $\delta$, $r_1$ and $r_2$, we have
\[
N_l(\delta, r_1, r_2) =
\begin{cases}
	0 &\mbox{if } \delta < 8k+4 \text{ or } \delta < r_1 + r_2, \text{ and}\\
	O(30 ^ \delta n^{2r_1} k ^ {6 r_2} / n^2) &\mbox{otherwise}.\\
\end{cases}
\]
\end{proposition}
\begin{proof}
We estimate the number of large $v$-templates $(A,h)$ of type $(\delta, r_1, r_2)$ by first estimating the number of ways in which we may choose $A$, and then estimating the number of ways in which we may choose $h$ once we are given $A$.

First, we may assume that $\delta \ge 8k+4$ since if $(A,h)$ is a large $v$-template, then $\partial [-k,k]^2 \subset A$ by definition. Second, we may also suppose that $\delta \ge r_1 + r_2$; indeed, by considering a northern most vertex of each quasiblock of $A$ for example, we observe that the number of quasiblocks of $A$ is at most the size of $A$, so the claimed bound holds trivially in the case where $\delta < r_1 + r_2$.

We now estimate the number of ways to choose $A$. Since $A$ must contain an edge incident to $(0,0)$ and must additionally be dual-connected, it follows from Proposition~\ref{component} that the number of choices for $A$ (even ignoring the restriction that $A$ has precisely $r_1 + r_2$ quasiblocks) is at most $4 (5e)^{\delta - 1} \le 15 ^ \delta$ as each edge of the square lattice is adjacent to six other edges of the square lattice in the planar dual of the lattice.

Next, we estimate the number of ways to choose $h$ for a given $A$. Once we fix an $A$ with $r_1 + r_2$ quasiblocks, it suffices to specify the image of one vertex from each quasiblock of $A$ under $h$ to completely specify $h$ since each quasiblock of $A$ is $h$-rigid. We count the number of ways to choose $h$ as follows. We first choose $r_1$ \emph{representative} quasiblocks in such a way that these quasiblocks all belong to different clusters, while ensuring that the quasiblock containing $(0,0)$ is one of these representatives; the number of ways to choose these representatives is at most
\[\binom{r_1 + r_2 - 1}{r_1 - 1} \le 2^{r_1 + r_2} \le 2^\delta.\]
Of course, since $h(0,0) = v$, this specifies the image of the quasiblock containing $(0,0)$. We then specify the image of a vertex (say the northernmost) from each of the remaining $r_1 -1$ representative quasiblocks; this may be done in $n^{2(r_1 - 1)}$ ways. Finally, we note that there are $O(k^6)$ choices for the image of one of the $r_2$ leftover quasiblocks. To see this, note that each leftover quasiblock belongs to the same cluster as one of the representative quasiblocks, so the image of such a leftover quasiblock must be at distance at most $(2k+1)^2$ from the image of one of the representative quasiblocks; the claimed bound follows since there are at most $(2k+1)^2$ points contained in the representative quasiblocks, and there are at most $(2d + 1)^2$ points at distance at most $d$ from any fixed point of the grid. Combining these estimates, we see that the number of choices for $h$ once we have specified $A$ is $O(2^\delta n^{2(r_1-1)} k ^ {6 r_2} )$.

It now follows that
\[N(\delta, r_1, r_2) = O(15^\delta 2^\delta n^{2(r_1-1)} k ^ {6 r_2}) = O(30 ^ \delta n^{2r_1} k ^ {6 r_2} / n^2). \qedhere\]
\end{proof}

To estimate the probability that a large $v$-template $(A,h)$ is feasible for $\JJ(n,q)$, we shall appeal to Proposition~\ref{f-cost} which gives us a bound for this probability in terms of $\gamma(h)$; recall that $\gamma(h)$ is the difference between the size of the vertex set of $\C{G}_h$ and the number of connected components of $\C{G}_h$, where $\C{G}_h$ is the constraint graph of $h$.

\begin{proposition}\label{ltemp-bound}
If $(A,h)$ is a large $v$-template of type $(\delta, r_1, r_2)$, then we have $\gamma(h) \ge \delta/ 20$ and $\gamma(h) \ge 2r_1 + r_2/2 - 2r_1/(2k+1)$.
\end{proposition}
\begin{proof}
We shall use Proposition~\ref{gamma-bound} to bound $\gamma(h)$ from below. We will estimate the size of both the vertex set and the edge set of $\C{G}_h$.

Since $\C{G}_h$ contains one edge for each $h$-split edge, it is easy to see that the edge set of $\C{G}_h$ has size at least $A^*$, so $ |E(\C{G}_h)| \ge|A^*|  = |A| - (8k + 4)$ as $|\partial [-k,k]^2| = 8k+4$. Now, since $A$ contains an edge incident to $(0,0)$, is dual-connected and also contains $ \partial [-k,k]^2$, we have $|A^*| \ge k$ and consequently, $|A| \ge 9k + 4$; it follows, provided $k$ is sufficiently large, that $ |E(\C{G}_h)| \ge |A| - (8k + 4) \ge |A|/10 = \delta/10$. We now conclude from Proposition~\ref{gamma-bound} that $\gamma(h) \ge |E(\C{G}_h)|/2 \ge \delta/20$.

To estimate the size of the vertex set of $\C{G}_h$, we begin with the following observation. First, if $\hat X$ is a quasiblock of $A$, then since $\hat X$ is $h$-rigid, there is a one-to-one correspondence between $\partial_e \hat X$ and $\partial_e h(\hat X)$. Next, note that each edge of $\partial_e \hat X$ is either an element of $A^*$ (and consequently $h$-split) or an element of $\partial [-k,k]^2$. It now follows that each edge of $\partial_e h(\hat X)$ that corresponds to an  edge of $\partial_e \hat X$ contained in $A^*$ must belong to the vertex set of $\C{G}_h$.

For a cluster $K$ of the template $(A,h)$ composed of the images of the quasiblocks $\hat X_1, \hat X_2, \dots, \hat X_m$, we write $S(K)$ for the set of edges between $\hat X_i$ and $\hat X_j$ for some $1 \le i < j \le m$ and $T(K)$ for the set $S(K) \cup \partial_e K$. First, it is clear that $S(K)$ and $\partial_e K$ are disjoint for each cluster $K$. Furthermore, it is also easy to see that if $K_1$ and $K_2$ are distinct clusters, then the sets $T(K_1)$ and $T(K_2)$ are disjoint. Let $T \subset \EDB([n]^2)$ denote the union of the sets $T(K)$, where $K$ runs over the $r_1$ clusters of $(A,h)$. From our earlier discussion, it follows that an edge of $T$ is a vertex of $\C{G}_h$ unless it corresponds to an edge in $\partial [-k,k]^2$. Consequently, we have $|V(\C{G}_h)| \ge |T| - (8k + 4)$.

We now use an isoperimetric argument to bound $|T|$ from below; we begin with following observation.

\begin{claim}\label{jcurve}
For a cluster $K$ of $(A,h)$ composed of the images of the quasiblocks $\hat X_1, \hat X_2, \dots, \hat X_m$, we have
\[
|T(K)| \ge 4\sqrt{|X_1| + |X_2| + \dots + |X_m|} + m - 1.
\]
\end{claim}
\begin{proof}
It immediately follows from the fact that $K$ corresponds to a connected component of size $m$ in the cluster graph of $(A,h)$ that $|S(K)| \ge m-1$. Next, while Proposition~\ref{p:iso} immediately tells us that
\[|\partial_e K| \ge 4\sqrt{|\hat X_1| + |\hat X_2| + \dots + |\hat X_m|},\] we may get a better estimate as follows. Note that since $\partial [-k,k]^2 \subset A$, the quasiblock $\hat X$ of $A$ associated with a connected component $X$ of $D([-k,k]^2, A)$ is in fact the vertex boundary of $X$. Therefore, it follows from the Jordan curve theorem that each quasiblock of $A$ must divide the plane into an exterior and an interior region. From the definition of a cluster, it follows that $h(\hat X_i)$ lies in the exterior of $h(\hat X_j)$ for all $i \ne j$. Consequently, it follows that $\partial_e K$ is in fact the external boundary of a set of size $|X_1| + |X_2| + \dots + |X_m|$; therefore, we have
\[|\partial_e K| \ge 4\sqrt{|X_1| + |X_2| + \dots + |X_m|}.\]
The claim follows since $S(K)$ and $\partial_e K$ are disjoint.
\end{proof}

By summing the bound from Proposition~\ref{jcurve} over the $r_1$ clusters of $(A,h)$, we obtain a bound of the form
\[ |T| \ge 4\sqrt{a_1} + 4\sqrt{a_2} + \dots + 4\sqrt{a_{r_1}} + r_2\]
for some collection of positive integers $a_1, a_2, \dots, a_{r_1}$ satisfying $a_1 + a_2 + \dots + a_{r_1} = (2k+1)^2$; this is immediate once we note that each connected component of $D([-k,k]^2,A)$ contributes precisely once to the bound in Proposition~\ref{jcurve} as we run over the clusters of $(A,h)$. We conclude, using convexity, that
\[ |T| \ge 4(r_1 - 1) + 4\sqrt{(2k+1)^2 - (r_1 - 1)} + r_2 \ge 4r_1 + r_2 + 8k+4 - \frac{4r_1}{(2k+1)}.\]
We know from Proposition~\ref{gamma-bound} that $\gamma(h) \ge |V(\C{G}_h)|/2 \ge |T|/2 - (4k+2)$; it now follows that $\gamma(h) \ge 2r_1 + r_2 - 2r_1/(2k+1)$.
\end{proof}

\subsection{Small templates}
We shall handle small templates using arguments similar to those used to deal with large templates; however, some small subtleties necessitate a slightly different approach to bookkeeping. If $(A,h)$ is small $v$-template, then it may well be the case that $|A|$ is small, so our estimates need to be capable of handling this; this cannot happen when $(A,h)$ is large since $\partial [-k,k]^2 \subset A$ in this case. On the other hand, if $(A,h)$ is small, then since $\partial [-k,k]^2 \cap A = \emptyset$, we do not need to worry about overcounting contributions from $\partial [-k,k]^2$ when estimating $\gamma(h)$. We will modify the arguments we used to deal with large templates slightly in order to balance these considerations.

Let $(A,h)$ be a small $v$-template. Since $A \cap \partial [-k,k]^2 = \emptyset$, it is easy to verify that the vertex boundary of $[-k,k]^2$ is contained in a single connected component of $D([-k,k]^2, A)$; we call the quasiblock corresponding to this connected component the \emph{boundary quasiblock} of $A$, and refer to the other quasiblocks of $A$ as \emph{non-boundary quasiblocks}.

We will need a slight modification of the notion of a `cluster' that distinguishes between the boundary quasiblock and non-boundary quasiblocks. Let $(A,h)$ be a small $v$-template and as before, for a quasiblock $\hat X \subset  [-k,k]^2$ of $A$, let $h(\hat X) \subset [n]^2$ denote the (rigid) image of $\hat X$ under $h$. Let us define the \emph{cluster graph} of $(A,h)$ to be the graph on the quasiblocks of $A$ where
\begin{enumerate}
\item two non-boundary quasiblocks $\hat X$ and $\hat Y$ are adjacent if there exists an edge of the square lattice between $h(x)$ and $h(y)$ for some $x \in \hat X$ and $y \in \hat Y$ and furthermore, this edge belongs to the external boundary of both $h(\hat X)$ and $h(\hat Y)$, and
\item the boundary quasiblock $\hat X$ and a non-boundary quasiblock $\hat Y$ are adjacent if there exists an edge of the square lattice between $h(x)$ and $h(y)$ for some $x \in \hat X$ and $y \in \hat Y$ and furthermore, this edge belongs to the internal boundary of $h(\hat X)$ and the external boundary of $h(\hat Y)$.
\end{enumerate}
A \emph{cluster} of $(A,h)$ is then a subset $[n]^2$ consisting of the images of all the quasiblocks in a connected component of the cluster graph; again, we call the cluster containing the image of the boundary quasiblock the \emph{boundary cluster} and refer to the other clusters as \emph{non-boundary clusters}.

As before, for non-negative integers $\delta$, $r_1$ and $r_2$, we say that a small $v$-template $(A,h)$ is of type $(\delta, r_1, r_2)$ if $|A| = \delta$, the number of quasiblocks of $A$ is $r_1 + r_2$, and the number of clusters of $(A,h)$ is $r_1$. Writing $N_s(\delta, r_1, r_2)$ for the number of small  $v$-templates of type $(\delta, r_1, r_2)$, we have the following estimate, the proof of which is identical to that of Proposition~\ref{ltemp-count}.

\begin{proposition}\label{stemp-count}
For non-negative integers $\delta$, $r_1$ and $r_2$, we have
\[N_s(\delta, r_1, r_2) = O(30 ^ \delta n^{2r_1} k ^ {6 r_2} / n^2). \eqno\qed\]
\end{proposition}

To estimate the probability that a small $v$-template $(A,h)$ is feasible for $\JJ(n,q)$, we will use the following.

\begin{proposition}\label{stemp-bound}
If $(A,h)$ is a small $v$-template of type $(\delta, r_1, r_2)$, then we have $\gamma(h) \ge \delta/2$ and $\gamma(h) \ge 2r_1 + r_2/2 + 1/2$.
\end{proposition}
\begin{proof}
As before, we will estimate the size of both the vertex set and the edge set of the constraint graph $\C{G}_h$.
	
Since $\C{G}_h$ contains one edge for each $h$-split edge, it is easy to see that edge set of $\C{G}_h$ has size at least $A^*$. Since $\partial [-k,k]^2 \cap A = \emptyset$, we have $A^* = A$, so $ |E(\C{G}_h)| \ge|A|  = \delta$. We now conclude from Proposition~\ref{gamma-bound} that $\gamma(h) \ge |E(\C{G}_h)|/2 \ge \delta/2$.
	
To estimate the size of the vertex set of $\C{G}_h$, we begin with the following observations. First, if $\hat X$ is a non-boundary quasiblock of $A$, then since $\hat X$ is $h$-rigid, there is a one-to-one correspondence between $\partial_e \hat X$ and $\partial_e h(\hat X)$; since each edge of $\partial_e \hat X$ an element of $A$ (and consequently $h$-split, as $A = A^*$), it follows that each edge of $\partial_e h(\hat X)$ must belong to the vertex set of $\C{G}_h$. Next, if $\hat X$ is the boundary quasiblock of $A$, then since each edge of $\partial_i \hat X$ is an element of $A$, it follows that each edge of $\partial_i h(\hat X)$ must belong to the vertex set of $\C{G}_h$.
	
For a non-boundary cluster $K$ of $(A,h)$ composed of the images of the non-boundary quasiblocks $\hat X_1, \hat X_2, \dots, \hat X_m$, we write $S(K)$ for the set of edges between $\hat X_i$ and $\hat X_j$ for some $1 \le i < j \le m$ and $T(K)$ for the set $S(K) \cup \partial_e K$; it is clear that $S(K)$ and $\partial_e K$ are disjoint, so $T(K)$ is in fact the disjoint union of these sets. For the boundary cluster $K$ composed of the images of the boundary quasiblock $X$ and non-boundary quasiblocks $\hat X_1, \hat X_2, \dots, \hat X_m$, we write $S(K)$ for the set of edges between $\hat X_i$ and $\hat X_j$ for some $1 \le i < j \le m$ and $T(K)$ for the set $S(K) \cup \partial_i h(\hat X)$; again, it is clear that $S(K)$ and $\partial_i h(\hat X)$ are disjoint and that $T(K)$ is the disjoint union of these sets. Finally, it is also easy to see that if $K_1$ and $K_2$ are distinct clusters, then the sets $T(K_1)$ and $T(K_2)$ are disjoint. As before, let $T \subset \EDB([n]^2)$ denote the union of the sets $T(K)$, where $K$ runs over the $r_1$ clusters of $(A,h)$. From our earlier observations, it follows that each edge of $T$ is a vertex of $\C{G}_h$; consequently, we have $|V(\C{G}_h)| \ge |T|$.
	
To bound $|T|$ from below, we first deal with non-boundary clusters.
	
\begin{claim}\label{nbcluster}
For a non-boundary cluster $K$ of $(A,h)$ composed of the images of the non-boundary quasiblocks $\hat X_1, \hat X_2, \dots, \hat X_m$, we have
$|T(K)| \ge 4 + (m - 1)$.
\end{claim}
\begin{proof}
It immediately follows from the fact that $K$ corresponds to a connected component of size $m$ in the cluster graph of $(A,h)$ that $|S(K)| \ge m-1$. Since the external boundary of a non-empty subset of the square lattice contains at least four edges, it follows that  $|\partial_e K| \ge 4$. The claim follows since $S(K)$ and $\partial_e K$ are disjoint.
\end{proof}

Next, we have the following estimate for the boundary cluster.
\begin{claim}\label{bcluster}
If the boundary cluster $K$ of $(A,h)$ is composed of the images the boundary quasiblock $\hat X$ and non-boundary quasiblocks $\hat X_1, \hat X_2, \dots, \hat X_m$, then $|T(K)| \ge 6 + (m - 1)$.
\end{claim}
\begin{proof}
As before, it is clear that $|S(K)| \ge m-1$. We claim that $|\partial_i h(\hat X)| \ge 6$. To see this, consider the connected component $X$ of $D([-k,k]^2, A)$ containing the vertex boundary of $[-k,k]^2$. Writing $X' = [-k,k]^2 \setminus X$, note we must have $|X'| \ge 2$, for if not, then $A$ must consist of precisely the four edges incident either to $(0,0)$ or one of its four neighbours. Note also that $\partial_e X' \subset \partial_i X$. Now since $X'$ contains at least two vertices, it is easily verified that $\partial_e X'$ contains at least six edges; consequently $|\partial_i h(\hat X)| = |\partial_i \hat X| = |\partial_i X| \ge 6$. The claim follows since $S(K)$ and $\partial_i h(\hat X)$ are disjoint.
\end{proof}
	
By summing the bound from Claim~\ref{nbcluster} over the $r_1 - 1$ non-boundary clusters of $(A,h)$ and then adding the bound from Claim~\ref{bcluster}, we obtain
\[ |T| \ge 4(r_1-1) + 6 + (r_2 - 1)  = 4r_1 + r_ 2 + 1.\]
We know from Proposition~\ref{gamma-bound} that $\gamma(h) \ge |V(\C{G}_h)|/2 \ge |T|/2$, so it follows from the above bound that $\gamma(h) \ge 2r_1 + r_2/2 + 1/2$.
\end{proof}

\subsection{Proof of the main lemma}

We are now in a position to prove Lemma~\ref{lnhd}.

\begin{proof}[Proof of Lemma~\ref{lnhd}.]
We shall show for any $v \in [n]^2$, using a union bound over all $v$-templates, that the probability that there exists a $v$-template that is feasible for $\JJ = \JJ(n,q)$ is $o(n^{-2})$ when $q \ge \CC n$; the lemma then follows from a union bound over the elements of $[n]^2$.

Fix a vertex $v \in [n]^2$. Let $E_l$ denote the event that there exists a large $v$-template that is feasible for $\JJ$, and let $E_s$ denote the event that there exists a small $v$-template that is feasible for $\JJ$.

First, we bound $\P(E_l)$ as follows. Consider the event $E_l (\delta, r_1, r_2)$ that there exists a large $v$-template of type $(\delta, r_1, r_2)$ that is feasible for $\JJ$. Of course, Proposition~\ref{ltemp-count} implies that $\P (E_l (\delta, r_1, r_2)) = 0$ either if $\delta < 8k + 4$ or if $\delta < r_1 + r_2$. Otherwise, from Propositions~\ref{ltemp-count} and~\ref{ltemp-bound} and the fact that $k = \lceil \log n \rceil$, we see that
\begin{align*}
\P (E_l (\delta, r_1, r_2)) &= O\left(\frac{30 ^ \delta n^{2r_1} k ^ {6 r_2} }{n^2 100^{\delta} n^{2r_1 + r_2/2 - 2r_1/(2k+1)}} \right) = O\left(\frac{30^\delta n^{2r_1/(2k+1)}  }{n^2 100^{\delta}} \right)\\
&= O\left(\frac{ e^{r_1}  }{n^2 3^{\delta}}\right) = O\left(\frac{ e^{\delta}  }{n^2 3^{\delta}}\right) = O\left(\frac{ (e/3)^{\log n}  }{n^2}\right) = O\left(n^{-1 - \log 3}\right).
\end{align*}
Now, since $1 + \log 3 > 2$ and $k = \lceil \log n \rceil$, we deduce from the above estimate that
\[ \P(E_l) = \sum_{\delta = 1}^{4(2k+1)^2}\sum_{r_1 = 1}^{(2k+1)^2} \sum_{r_2 = 1}^{(2k+1)^2} \P (E_l (\delta, r_1, r_2)) = O\left(4(2k+1)^6 n^{-1 - \log 3}\right) =  o(n^{-2}).\]

Next, we bound $\P(E_s)$ as follows. Consider the event $E_s (\delta, r_1, r_2)$ that there exists a small $v$-template of type $(\delta, r_1, r_2)$ that is feasible for $\JJ$. From Propositions~\ref{stemp-count} and~\ref{stemp-bound} and the fact that $k = \lceil \log n \rceil$, we see that
\[
\P (E_s (\delta, r_1, r_2)) = O\left(\frac{30 ^ \delta n^{2r_1} k ^ {6 r_2} }{n^2 10^{20\delta} n^{2r_1 + r_2/2 + 1/2}} \right) = O\left(\frac{30^\delta  }{n^{2+1/2} 100^{\delta}} \right) = O\left(n^{-2 - 1/2}\right).
\]
As before, since $k = \lceil \log n \rceil$, we deduce from the above estimate that
\[
\P (E_s) = \sum_{\delta = 1}^{4(2k+1)^2}\sum_{r_1 = 1}^{(2k+1)^2} \sum_{r_2 = 1}^{(2k+1)^2} \P (E_s (\delta, r_1, r_2)) =  O\left(4(2k+1)^6 n^{-2 - 1/2}\right) =  o(n^{-2}).\]

It follows that the probability that there exists $v$-template that is feasible for $\JJ$ is $o(n^{-2})$; the lemma follows from a union bound over the vertices of the grid.
\end{proof}

\section{Proof of the $1$-statement}\label{s:upper}
In this section, we prove the $1$-statement in Theorem~\ref{reconthm}. We proceed roughly as in~\citep{feige} by first assembling the `central bulk' of a random jigsaw using Lemma~\ref{lnhd}, and then extending this assembly to the `periphery' in a fairly straightforward fashion; our arguments will however require a bit more work than the one in~\citep{feige} since we have fewer colours to work with.

\begin{proof}[Proof of the $1$-statement in Theorem~\ref{reconthm}]
Suppose that $q \ge \CC n$, let $\JJ = \JJ(n,q)$ and, as in Section~\ref{s:nhd}, let $k = \lceil \log n\rceil$. To prove the $1$-statement, we shall describe an algorithm that reconstructs $\JJ$ from its deck $\DD(\JJ)$ with high probability.

We begin by addressing the possibility of tiles occurring with multiplicity greater than one in $\DD(\JJ)$. Let $X_1$ denote the number of pairs $(u, v) \in ([n]^2)^2$ with $\JJ_u = \JJ_v$, and let $X_2$ denote the number of pairs $(u, v) \in ([n]^2)^2$ with $\JJ_u = \JJ_v$ such that $u$ and $v$ are additionally at distance at most two from each other. We then observe the following.

\begin{claim}\label{doubles}
	$\E[X_1] \le 1 $ and $\E[X_2] = o(1)$.
\end{claim}
\begin{proof}
	The claim follows immediately from noting that $\E[X_1] = n^4q^{-4}$ and that $\E[X_2] = O(n^2 q^{-4})$.
\end{proof}

%First, let $X_1$ denote the number of triples $(u,v,w) \in ([n]^2)^3$ with $\JJ_u = \JJ_v = \JJ_w$; we have the following simple estimate.

%\begin{claim}\label{triples}
%$\E[X_1] = o(1)$.
%\end{claim}
%\begin{proof}
%The expected number of triples $(u,v,w) \in ([n]^2)^3$ such that each of $u$, $v$ and $w$ are at distance at least two from each other with $\JJ_u = \JJ_v = \JJ_w$ is $O(n^6q^{-8})$. By similarly considering triples where precisely two vertices are close together, and triples where all three vertices are close together, it may be verified that $\E[X_1] = O(n^2q^{-6} + n^4q^{-7} + n^6q^{-8})$; the claim follows since $q \ge \CC n$.
%\end{proof}

Let us now record some properties that are possessed by $\JJ$ with high probability.

\begin{enumerate}[(A)]
\item\label{window} There exist no non-trivial $v$-windows with respect to $\JJ$ for any $v \in [n]^2$;  this follows from Lemma~\ref{lnhd}.
%\item\label{mul3} The multiplicity of each element of $\DD(\JJ)$ is either one or two; this follows from Claim~\ref{triples} and Markov's inequality.
\item\label{mul2} The number of vertices $v \in [n]^2$ such that the tile $\JJ_v$ has multiplicity greater than one in $\DD(\JJ)$ is at most $\log n$; this follows from Claim~\ref{doubles} and Markov's inequality.
\item\label{mul2close} If $\JJ_u = \JJ_v$ for some $u, v \in [n]^2$	, then the distance between $u$ and $v$ is at least three; this again follows from Claim~\ref{doubles} and Markov's inequality.
\end{enumerate}

We first show how one may reconstruct a large subgrid of $\JJ$ from $\DD(\JJ)$ with high probability; we do this by showing how one may perform this reconstruction assuming that $\JJ$ satisfies~\eqref{window},~\eqref{mul2} and~\eqref{mul2close}. To this end, we proceed by building a labelled, directed graph $H$ on $\DD(\JJ)$ to encode the relative positions of the tiles in the jigsaw. In what follows, a component of the directed graph $H$ will mean a connected component of the underlying undirected graph.

First, we consider every tile $t \in \DD(\JJ)$ which occurs in the deck with multiplicity one. For such a tile $t$, we consider all possible subsets of $(2k+1)^2$ tiles that include $t$, and for each such set, we consider all possible arrangements of this set of tiles on the grid $[-k,k]^2$ with $t$ being placed at $(0,0)$. Finally, for each such arrangement that is feasible, we record the tuple $(t,t_1, t_2, t_3, t_4)$, where $t_i$ is the tile placed at $e_i$ in this arrangement for $1 \le i \le 4$. Now, for each recorded tuple $(t, t_1, t_2, t_3, t_4)$, we add an edge directed from $t$ to $t_i$ labelled $e_i$ in $H$ if the tile $t_i$ also occurs with multiplicity one in the deck.

It follows from~\eqref{window} that if there exists a directed edge from a tile $t$ to a tile $t'$ labelled $e_i$ in $H$, then it must be the case that $t = \JJ_v$ and $t' =\JJ_{v'}$, where $v, v' \in[n]^2$ are vertices such that $v' = v + e_i$. Consequently, each component of $H$ describes the relative positions of the tiles in that component in $\JJ$; in other words, for any two tiles $t = \JJ_v$ and $t' = \JJ_{v'}$ that belong to the same component in $H$, we may determine $v - v'$ using $H$.

From~\eqref{mul2close} and Propostion~\ref{p:3sep}, we deduce that the tiles of $\JJ$ coming from the central $(n-2k) \times (n-2k)$ subgrid of $[n]^2$ which furthermore appear with multiplicity one in $\DD(\JJ)$ all belong to the same component of $H$; it follows from~\eqref{mul2} that this component contains at least $(n-2k)^2 - \log n > n^2 / 2$ tiles, and is consequently the unique largest component of $H$.

Next, we fill in the `holes' in the largest component of $H$ as follows. We know that we may determine, up to translation, the positions on the square lattice of all the tiles in a given component of $H$; we fix an arrangement of the tiles in the largest component by placing one of these tiles at the origin and the other tiles at their appropriate positions relative to the origin. Suppose that there is no tile at some position $x \in \Z^2$ in this arrangement, but that there is a tile at each of the four positions neighbouring $x$. Now, the tiles in the positions neighbouring $x$ uniquely determine the missing tile at $x$, and since all pairs of adjacent tiles in $H$ come from adjacent positions in $[n]^2$, it follows that such a missing tile must be an isolated vertex of $H$. Once we add each such missing tile to the largest component of $H$ (by adding in the appropriately labelled directed edges), it follows from~\eqref{mul2close} that the largest component of $H$ contains each tile of $\JJ$ coming from the central $(n-2k-2) \times (n-2k-2)$ subgrid of $[n]^2$. Let $S_H$ denote the largest square subgrid contained in the largest connected component of $H$ at this juncture; we know from the above discussion that with high probability, $S_H$ is a fully-assembled $s \times s$ subgrid of $\JJ$ with $s\ge n-2k - 2$.

We now finish the proof by showing that $\JJ$ has the following property with high probability: given any fully-assembled $m \times m$ subgrid $M$ of $\JJ$ with $m \ge n-2k-2$, there is a unique way to assemble the tiles not in $M$ around $M$ to produce a feasible assembly of tiles on an $n\times n$ grid; of course, this final assembly of tiles must then coincide with $\JJ$.

Let us now describe an extension procedure that, with high probability, extends a given large fully-assembled subgrid $M$ uniquely to $\JJ$ using the tiles not in $M$. This extension procedure will proceed by repeatedly extending $M$, first upwards, then downwards, then to the left and finally to the right, adding an entire row or column of tiles at each step (thus ensuring that $M$ remains a subgrid at each stage). Suppose first that we wish to add a row of tiles to the top of $M$. Let $M'$ denote the set of tiles $t$ in the top row of $M$ not located at one of the two corners,  and let $M''$ denote the set of two tiles at the top corners of $M$. For each $t \in M'$, we record all triples $(t', t'_l, t'_r)$ of tiles from the deck (not already in $M$) such that we may feasibly place $t'$ above $t$, $t'_l$ to the immediate left of $t'$, and $t'_r$ to the immediate right of $t'$. We then proceed as follows.
\begin{enumerate}
\item If no such feasible triple of tiles $(t', t'_l, t'_r)$ exists for some tile $t \in M'$, then we stop attempting to extend $M$ upwards and change directions.
\item If there exist two distinct choices for the tile $t'$ over all recorded feasible triples $(t', t'_l, t'_r)$ for some tile $t \in M'$, then we abort.
\item If there exists a single choice for $t'$ (though potentially more than one choice for $t'_l$ and $t'_r$) over all recorded feasible triples for each tile $t \in M'$, then we add a new row of tiles to the top of $M$ by first placing $t'$ above $t$ for each tile $t \in M'$. We then check if there exists a unique way to place two tiles (that are not already in $M$) feasibly above the two tiles in $M''$, and if so, we finish adding a new row to the top of $M$ by placing these two tiles in place; if we either cannot find such a pair of tiles, or if multiple choices exist for this pair, then we again abort.
\end{enumerate}
Assuming that we have not aborted at any stage, we then continue to add rows to the top of $M$ until we are forced to change directions, and we then similarly extend $M$ downwards, to the left and finally to the right.

To bound the probability that this extension procedure fails to uniquely reconstruct $\JJ$ from some large fully-assembled subgrid, we need to define two events. It will be convenient to first have some notation. Let $B_1 \subset [n]^2$ denote the set of vertices not contained in the central $(n - 4k - 4) \times (n-4k-4)$ subgrid of $[n]^2$, and let $B_2 \subset B_1$ denote the set of vertices in the four $(2k+2) \times (2k+2)$ subgrids at the four corners of $B_1$.

We first address the possibility of `failing in a corner' when extending a large subgrid. Let $E_1$ denote the event that there exists a pair $(u,v)$ with $u \in B_2$ and $v \in B_1$ such that some two edges incident to $u$ receive the same two colours  under $\JJ$ as  some two edges incident to $v$. We then have the following estimate.

\begin{claim}\label{corner}
$\P(E_1) = o(1)$.
\end{claim}
\begin{proof}
Let $Y_1$ denote the number of pairs $(u,v)$ which satisfy the conditions of the event $E_1$. It is easy to see that
\[
\E[Y_1] = O(nk^3 q^{-2} + k^2q^{-1}) = o(1);
\]
the claim follows from Markov's inequality.
\end{proof}

Next, we address the possibility of `failing in the bulk of a row or column' when extending a large subgrid. Let $E_2$ denote the event that there exists a quadruple $(u,v,v', v'')$, where $u, v, v', v'' \in B_1$ and $u$ is not one of the four corners of $[n]^2$, such that either $v \ne u + e_1$ and the map $f\colon [-1,1] \times [0,1] \to [n]^2$ defined by $f(-1,0) = u+e_4$, $f(0,0) = u$, $f(1,0) = u+e_2$, $f(-1,1) = v'$, $f(0,1) = v$ and $f(1,1) = v'')$ is feasible for $\JJ$, or such that the quadruple satisfies an analogous condition with respect to one of the three other directions. We then have the following estimate.

\begin{claim}\label{middle}
	$\P(E_2) = o(1)$.
\end{claim}
\begin{proof}
Let $Y_2$ denote the number of quadruples $(u,v, v',v'')$ which satisfy the conditions of the event $E_b$. We may verify (after a somewhat tedious case analysis) that
\[
\E[Y_2] = O((nk)^4q^{-5} + (nk)^3q^{-4} + (nk)^2q^{-3}) = o(1);
\]
the claim follows from Markov's inequality.
\end{proof}

If neither $E_1$ nor $E_2$ occurs, then it is easily seen by induction that our extension procedure extends any fully assembled $m \times m$ subgrid $M$ with $m \ge n - 2k -2 $ uniquely to $\JJ$ using the tiles not already in $M$. It follows that we may extend $S_H$ uniquely to $\JJ$ with high probability, proving the theorem.
\end{proof}
\section{Conclusion}\label{s:conc}
We conclude by reminding the reader of Conjecture~\ref{mainconj} which asserts that the answer to the question of whether $\JJ(n,q)$ is reconstructible exhibits a sharp transition at $q \approx n / \sqrt{e}$. Here, we have established the $0$-statement in Conjecture~\ref{mainconj} using a simple counting argument. We have also proved the $1$-statement in this conjecture for all $q \ge Cn$, where $C >0$ is some absolute constant. As mentioned earlier, it is possible to use our methods to show that we may actually take $C$ as above to be any constant strictly greater than $1$: roughly speaking, our estimates for the number of templates in the `small edge boundary' regime are very crude, and it is possible to do significantly better in this regime using stability results (see~\citep{ellis}, for example) for the isoperimetric inequality in $\Z^2$. However, showing that we may actually take $C$ as above to be any constant strictly greater than $1/\sqrt{e}$ appears to be completely out of the reach of our methods; we expect new ideas will be required to settle this problem.

Of course, one could also ask for the size of the window in the sharp transition predicted by Conjecture~\ref{mainconj}. By repeating the proof of the $0$-statement of Theorem~\ref{mainconj} with more careful estimates, we are led to the following refinement of Conjecture~\ref{mainconj} whose $0$-statement again follows from our counting argument.
\begin{conjecture}\label{refinement}
Let $q= q(n) = n/\sqrt{e} + \log n + \alpha(n)$. As $n \to \infty$, we have
\[
\P \left( \JJ(n,q) \text{ is reconstructible} \right) \to
\begin{cases}
1 &\mbox{if } \alpha(n) \to \infty, \text{ and}\\
0 &\mbox{if } \alpha(n) \to -\infty.\\
\end{cases}
\]
\end{conjecture}

Finally, it would be of interest to investigate higher-dimensional analogues of the problem considered here. For example, it would be interesting to decide if the analogous $d$-dimensional problem of reconstructing a random $q$-colouring of (the edges of) $[n]^d$ from its deck exhibits a sharp threshold at $q \approx n/ e^{1/d}$ for each $d\ge3$.
\section*{Acknowledgements}
The first and second authors were partially supported by NSF grant DMS-1600742, and the second author also wishes to acknowledge support from EU MULTIPLEX grant 317532. Some of the research in this paper was carried out while the first author was visiting the Isaac Newton Institute for Mathematical Sciences at the University of Cambridge; the first author is grateful for the hospitality of the Institute.

\bibliographystyle{amsplain}
\bibliography{jigsaw_recon}

\end{document}